\documentclass[10pt, reqno]{amsart}
\usepackage{amssymb}
\usepackage{latexsym}
\usepackage{amsmath}
\usepackage{amsthm}
\usepackage{amsfonts}
\usepackage{mathrsfs}
\usepackage{enumerate}
\usepackage{verbatim}
\usepackage[all,cmtip]{xy}
\usepackage[colorlinks=true,linkcolor=blue,citecolor=blue]{hyperref}

\numberwithin{equation}{section}

\newcommand{\intbar}{-\mkern-19mu \int}
\newcommand{\bp}[1]{{}^\partial\mkern-4mu #1}
\newcommand{\inlinetb}{-\mkern-16mu \int}

\newtheorem{theorem}{Theorem}[section]
\newtheorem*{theorem*}{Theorem}

\newtheorem{lemma}[theorem]{Lemma}
\newtheorem*{lemma*}{Lemma}

\newtheorem{corollary}[theorem]{Corollary}
\newtheorem*{corollary*}{Corollary}

\newtheorem{proposition}[theorem]{Proposition}
\newtheorem*{proposition*}{Proposition}

\newtheorem*{claim*}{Claim}

\theoremstyle{definition}
\newtheorem{definition}[theorem]{Definition}

\newtheorem{remark}[theorem]{Remark}

\theoremstyle{remark}

\newcommand{\diff}{\textup{d}}

\newcommand{\sflow}{\textup{sf}}

\newcommand{\vol}{\textup{vol}}

\newcommand{\End}{\textup{End}}

\newcommand{\cl}{\textup{Cl}}

\newcommand{\ahat}{\hat A}

\newcommand{\reta}{\xi}

\newcommand{\twD}{\mathfrak D}
\newcommand{\ptwD}{\bp \mathfrak D}

\newcommand{\tr}{\textup{tr}}
\newcommand{\Tr}{\textup{Tr}}

\newcommand{\Str}{\textup{Str}}
\newcommand{\str}{\textup{str}}

\newcommand{\bpo}{{}^{\textup b}\Psi}
\newcommand{\btr}{{}^{\textup b}{\textup{Tr}}}
\newcommand{\bstr}{{}^{\textup b}{\textup{Str}}}
\newcommand{\ptr}{{}^{\partial}{\textup{Tr}}}
\newcommand{\pstr}{\textup{Str}}

\newcommand{\pcon}{{}^{\partial}\mathbb A}

\newcommand{\bnorm}{{}^{\textup b}\|}

\newcommand{\ldb}{\ {}^{\textup b}\langle\langle} 
\newcommand{\rdb}{\rangle\rangle}
\newcommand{\bket}{\ {}^{\textup b}\langle}

\newcommand{\indo}{\textit{I}} 

\newcommand{\bcon}{\mathbb A}
\newcommand{\bnab}{{}^{\text b}\nabla}

\newcommand{\Ch}{\textup{Ch}}
\newcommand{\sch}{ \widetilde{\textup{Ch}}}

\newcommand{\bch}{{}^{\text b}\textup{Ch}}

\begin{document}
	
\title{The Odd Dimensional Analogue of a Theorem of Getzler and Wu}

\author{Zhizhang Xie}
\address{Department of Mathematics, The Ohio State University,
Columbus, OH, 43210-1174, USA}
\email{xiezz@math.ohio-state.edu}	
		
	
\begin{abstract}
	 We prove an analogue for odd dimensional manifolds with boundary, in the $b$-calculus setting, of the higher Atiyah-Patodi-Singer index theorem by Getzler and by Wu, thus obtain a natural counterpart of the eta invariant for even dimensional closed manifolds.
\end{abstract}
\thanks
{The author was partially supported by the US National Science Foundation awards no. DMS-0969672.}	

\maketitle

\section*{Introduction}

The goal of this paper to prove an analogue for odd dimensional manifolds with boundary of the higher Atiyah-Patodi-Singer index theorem of Getzler \cite{EG93b} and  Wu \cite{FW93}.  For notational simplicity, we will restrict the discussion mainly to spin manifolds. However all results can be straightforwardly extended to general manifolds, with appropriate modification.
 
Suppose $N$ is an odd dimensional spin manifold with boundary and carries an exact $b$-metric \cite{RM93}, cf. Section $\ref{sec:pre}$. For $g\in U_k(C^\infty(N))$ a unitary over $N$, let $\Ch_\bullet(g)$ (resp. $\Ch_\bullet^{\textup{dR}} (g)$) be the Chern character of $g$ in entire cyclic homology of $C^\infty(N)$ (resp. de Rham cohomology of $N$). In the following, $\inlinetb_N$ is the regularized integral on $N$ with respect to its $b$-metric (see Section \ref{sec:pre}) and $\ahat (N)$ is the $\ahat$-genus form of $N$. Let $D$ be the Dirac operator on $N$ and $\bp D$ be its restriction to the boundary $\partial N$. Denote the higher eta cochain of $\bp D$ by $\eta^\bullet(\bp D)$ , introduced by Wu \cite{FW93}.
\begin{theorem*}
 Let $N$ be an odd dimensional spin manifold with boundary. Endow $N$ with an exact $b$-metric and let $D$ be its associated Dirac operator. Assume $\bp D$ is invertible. For $g\in U_k(C^\infty(N))$ a unitary over $N$, if $\|[\bp D, g]\| < \lambda$ where $\lambda$ the lowest nonzero eigenvalue of $|\bp D|$, then
\begin{equation}\label{eq:main}
\sflow ( D, g^{-1} D g) = \ \intbar_{N} \ahat(N) \wedge \Ch_\bullet^{\textup{dR}}(g) + \left\langle \eta^\bullet(\bp D) , \Ch_\bullet (\bp g) \right\rangle.
\end{equation}

\end{theorem*}
Here $\sflow (D, g^{-1} D g)$ is the spectral flow of the path $D_u = (1-u) D + u g^{-1} D g $ with $u\in [0, 1]$ (see Section $ \ref{sec:sflow}$). In order for  $\sflow(D, g^{-1} D g) $ to be well-defined, the  infimum of the essential spectrum $\inf \textup{spec}_{\textup{ess}}(|D_u|)$ of $ |D_u|$ has to be greater than zero for each $u$. The latter condition is fulfilled if and only if  the restriction $D_u$ to the boundary $\partial N$ is invertible for each $u$. Thus the almost flatness condition $\|[\bp D, g]\| < \lambda$ ensures that $\sflow(D, g^{-1}Dg)$ is well-defined.

We consider the $b$-analogue $\bch^\bullet (D_t)$ of the odd Chern character by Jaffe-Lesniewski-Osterwalder \cite{JLO88}, cf. Section $\ref{sec:JLO}$. The theorem is proved by interpolating between the limit of $\bch^\bullet( D_t) $ as $t\to \infty$ and its limit as $t\to 0$, where $D_t  = t D$. In fact, the limit at $t = \infty$ does not exist in general. However, when evaluated at $\Ch_\bullet(g)$ with $g$ satisfying the almost flat condition above,  the limit of $\bch^\bullet(D_t)$ as $t\to \infty$ does exist and gives the spectral flow $\sflow(D, g^{-1}D g)$.  To prove this, i.e. the equality
\begin{equation}\label{eq:liminfty}
 \lim_{t\to \infty} \langle \bch^{\bullet}( D_t), \Ch_\bullet(g)\rangle = \sflow (D, g^{-1} D g),
\end{equation}
we first show (see Proposition $\ref{prop:sflow}$) that
\begin{equation}\label{eq:first}
  \sflow(D, g^{-1} D g) = \lim_{\varepsilon \to \infty}\frac{\varepsilon}{\sqrt \pi}\int_0^1 \btr( \dot{D}_u  e^{-\varepsilon^2 D_u^2}) du. 
\end{equation}
This is a generalization to the $b$-calculus setting of Getzler's spectral flow formula for closed manifolds, cf. \cite[Corollary 2.7]{EG93}. Once we show Eq. $\eqref{eq:first}$, the proof of Eq. $\eqref{eq:liminfty}$ reduces to
\begin{equation}\label{eq:sec}
\lim_{t \to \infty}\frac{t}{\sqrt \pi}\int_0^1 \btr( \dot{D}_u  e^{-t^2 D_u^2}) du =  \lim_{t\to \infty} \langle \bch^{\bullet}( D_t), \Ch_\bullet(g)\rangle.
\end{equation}
In turn, to verify this, we consider a multiparameter version of
the Chern character $\Ch( \bcon)$ of the superconnection $\bcon$ (see \cite{EG93}, also Section $\ref{sec:large}$ below, for the precise definition). 
Each side of Eq. $\eqref{eq:sec}$ corresponds to one term in the formula obtained by applying Stokes theorem to $\Ch(\bcon_t)$ for each fixed $t$. We then show the vanishing of the rest of the terms as $t\to \infty$, hence prove the validity of Eq. $\eqref{eq:sec}$, cf. Section $\ref{sec:large}$. The rest of the proof follows along the lines of Getzler's even counterpart, cf. \cite{EG93b}.  

Due to the fact that $\btr$ is not a trace, $\bch^\bullet(D_t)$ is not a closed cochain. The integral of its boundary from $0$ to $\infty$ gives the odd eta cochain $\eta^\bullet(\bp D)$ on the right hand side of $\eqref{eq:main}$. 
As a corollary of the main theorem, by comparing Eq. $\eqref{eq:main}$ with Dai-Zhang's Toeplitz index formula for odd dimensional manifolds with boundary \cite{XD-WZ06}, we obtain
\[ \left\langle \eta^\bullet(\bp D) , \Ch_\bullet (\bp g) \right\rangle = \eta(\partial N, \bp g) \mod \mathbb Z \]
where $\eta(\partial N, \bp g)$ is the eta invariant of Dai-Zhang. This equality provides more evidence for the naturality of the Dai-Zhang eta invariant for even dimensional closed manifolds. 

An outline of this article is as follows. In Section $ \ref{sec:pre}$, we recall some facts from $b$-calculus on manifolds with boundary and Chern characters in cyclic homology. In Section $\ref{sec:JLO}$, we define a $b$-analogue of the JLO Chern character and prove its entireness. Then we state our main theorem (Theorem $\ref{thm:main}$ below) in Section $ \ref{sec:main}$. We prove the main step of the proof to the main theorem in Section $ \ref{sec:sflow}$ and $ \ref{sec:large}$.

\subsection*{Acknowledgements}
I am grateful to Henri Moscovici for his continuous support and advice. I also want to thank Matthias Lesch for many helpful suggestions.

\section{Preliminaries}\label{sec:pre}
Throughout the paper, we denote by $\cl_q(\mathbb C)$ the complex Clifford algebra with odd generators $c_i$, $1\leq i \leq q$ and relations 
\[c_ic_j + c_jc_i = -2 \delta_{ij}.\]
This is a $\mathbb Z_2$-graded $\ast$-algebra with $c_i^\ast = - c_i$.

\subsection{Manifolds with Boundary and b-metrics}

Let $M$ be an odd dimensional spin manifold with boundary. We fix a Riemannian metric, say $w$, and a spin structure on $M$. Furthermore, we assume the Riemannian metric is of product type near the boundary, that is, on $[0, \varepsilon)_{x}\times \partial M$ a collar neighborhood of $\partial M$, it takes the form
\[ w = (dx)^2 + h  \]
where $h$ is the Riemannian metric on $\partial M$. Denote by $\widehat M$ the manifold obtained by attaching an infinite cylinder $(-\infty, 0]\times \partial M$ to $M$ along $\partial M$:
\[ \widehat M = (-\infty, 0]\times \partial M \cup_{\partial M} M\]  
The Riemannian metric $M$ extends naturally to a Riemannian metric on $\widehat M$, still denoted by $w$. 

Notice that $(\widehat M , w)$ is isometric to a standard $b$-manifold, that is, a manifold with boundary carrying a $b$-metric. To see this, one performs the change of variable $x\mapsto r= e^x$ on the cylindrical end. This replaces $(-\infty, 0]_x\times \partial M$ by a compact cylinder $[0, 1]_r\times \partial M$. Moreover, the metric $w$ induces a metric on $N = [0, 1]\times \partial M \cup_{\partial M} M$ under the coordinate change. In particular, the induced metric restricted on $[0, 1]_r\times \partial M$ takes the form
	 \[(\frac{dr}{r})^2 + h \]
which is an exact $b$-metric on $N$, cf. \cite{RM93} and \cite{PL05}. Unless otherwise specified, all $b$-metrics in this paper are assumed to be exact.
 	
\subsection{Clifford Modules and Dirac Operators}\label{sec:getzler}

Consider $N = [0, 1]_r\times \partial M \cup_{\partial M} M$ with an exact $b$-metric. 
The set of  $b$-vector fields, that is, vector fields on $N$ tangential to $\partial N$, is closed under Lie bracket. By Swan-Serre Theorem, such vector fields are smooth sections of a vector bundle ${}^b TN$ over $N$, called the $b$-tangent bundle of $N$, cf. \cite[Lemma 2.5]{RM93}. We denote its dual bundle, the $b$-cotangent bundle,  by ${}^b T^\ast N$.  By a Clifford module over $N$ of degree $q$, we mean a $\mathbb Z_2$-graded Hermitian vector bundle $\mathcal E$ over $N$ with commuting graded $\ast$-actions of the Clifford algebra $\cl_q(\mathbb C)$ and the Clifford bundle $\cl({}^b T^\ast N)$, cf. \cite{EG93b}. 

The spinor bundle $\mathcal S$ of $N$ naturally induces a Clifford module of degree $1$ as follows. Each $\omega \in \Gamma(N , \cl({}^b T^\ast N))$ acts on $\mathcal S \otimes \mathbb C^{1|1} $ by 
$ \begin{pmatrix} 0 & c(\omega) \\ c(\omega) & 0 \end{pmatrix}$
and the generator $e_1$ of $\cl_1(\mathbb C)$ acts by $ \begin{pmatrix} 0 & 1 \\ -1 & 0 \end{pmatrix} $, 
where $\mathbb C^{1|1}= \mathbb C^+ \oplus \mathbb C^-$ is $\mathbb Z_2$-graded.

\subsection{b-Norm} 
In this subsection, we introduce a $b$-norm on $C^\infty_{\textup{exp}}(\widehat M)$. We shall use this $b$-norm to prove the entireness of the $b$-JLO Chern character in section $\ref{sec:JLO}$. Here  $C^\infty_{\textup{exp}}(\widehat M)$ is the space of smooth functions on $\widehat M$ which expands exponentially on the infinite cylinder $(-\infty, 0]_x\times \partial M$, cf. \cite{PL05}. A smooth function $f\in C^\infty(\widehat M) $ expands exponentially on $(-\infty, 0]_x\times \partial M$ if 
\[ f(x, y) \sim \sum_{k=0}^\infty e^{kx}f_k(y)  \]
for $(x, y) \in (-\infty, 0]_x\times \partial M$,  where $f_k(y) \in C^\infty(\partial M)$ for each $k$.  More precisely, we have
\[ f(x, y) - \sum_{k=0}^{N-1} e^{kx}f_k(y) = e^N R_N(x, y)  \]
where all derivatives of $R_N(x, y)$ in $t$ and $y$ are bounded. 

\begin{remark}
Notice that $C^\infty_{\textup{exp}}(\widehat M)$ becomes exactly $C^\infty(N)$ if one performs the change of variable $x \to e^{x}$ on the cylindrical end. 
\end{remark}

On $(-\infty, 0]_x\times \partial M$, we write
\begin{equation*}
	 a = a_c + e^x a_{\infty}
\end{equation*}
for $a\in C^\infty_{\textup{exp}}(\widehat M)$, with $a_c, a_\infty\in C^\infty(\widehat M)$ and $a_c$ constant with respect to $x$.
We define a norm on $C^\infty_{\textup{exp}}(\widehat M)$ by
\[  \bnorm a\| := \|a\|_1 + 2 \|a_\infty\|_1  \] 
where $\|a\|_1$ is the usual $C^1$-norm of $a$ and $\|a_\infty\|_1$ is the $C^1$-norm of $a_\infty$. 
\begin{lemma}
$\bnorm \cdot \|$ is a well-defined multiplicative norm.
\end{lemma}
\begin{proof}
Note that  $(a+b)_{\infty} = a_\infty + b_\infty$ and $(ab)_\infty = a_c b_\infty + a_\infty b_c + e^t a_\infty b_\infty$. So it is clear that 
\begin{align*}
	 \bnorm \lambda a\| & = |\lambda|\cdot  \bnorm a \| \\
\bnorm a + b \| &\leq \bnorm a \| + \bnorm b \|
\end{align*}
To prove the norm is multiplicative, we first notice that $\|a_c \| \leq \|a\| $ , $\|da_c \| \leq \|da\|$ and 
\[ d(e^x a_\infty b_\infty) = (e^xdx) a_\infty b_\infty + e^x d(a_\infty b_\infty) .\]
 Thus we have 
\begin{align*}
& 2\| a_c b_\infty + a_\infty b_c + e^x a_\infty b_\infty\|_1 \\
 & = 2 \|  a_c b_\infty + a_\infty b_c + e^x a_\infty b_\infty \| + 2 \|d( a_c b_\infty + a_\infty b_c + e^x a_\infty b_\infty) \| \\
 & \quad + 2 \|a\| \cdot \|d b_\infty\| + 2 \|da_\infty\| \cdot \|b\|\\
& \leq 2 \|a\|_1 \cdot \|b_\infty\|_1 + 2 \|a_\infty\|_1 \cdot \|b\|_1 + 4 \|a_\infty\|_1 \cdot \|b_\infty\|_1
\end{align*}
By applying the inequality $\|ab\|_1 \leq (\|a\|+ \|da\|) (\|b\| + \|db\|) $, we obtain
\begin{align*}
\bnorm ab \| & \leq \bnorm a\| \cdot  \bnorm b\|.  
\end{align*}

\end{proof}

\subsection{b-Trace}

For $f\in C^\infty_{\textup{exp}}(\widehat M)$, we have 
\[ f = f_c + e^x f_\infty \]
on the cylindrical end $(-\infty, 0]\times \partial M$, where $f_c$ is constant with respect to $x$ .
\begin{definition} The regularized integral of $f\in C^\infty_{\textup{exp}}(\widehat M)$ with respect to the $b$-metric is defined to be
 \[ \intbar_{\widehat M} f \ d\vol := \int_{M} f|_M  \ d\vol+ \int_{(-\infty, 0]\times \partial M} e^xf_\infty \ d\vol. \]
\end{definition}

For $A\in \bpo^{-\infty}(\widehat M, \mathcal V) $, let  $K_{A}$ be its Schwartz kernel and $K_A|_\Delta$ the restriction of $K_A$ to the diagonal $\Delta \subset \widehat M\times \widehat M$. Then the fiberwise trace of $K_A|_\Delta$, denoted by $\tr(K_A|_\Delta)$, is a function in $C^\infty_{\textup{exp}}(\widehat M)$, cf. \cite{PL05}. We define the $b$-trace of $A\in \bpo^{-\infty}(\widehat M, \mathcal V)  $ to be 
	\[ \btr(A) := \intbar_{\widehat M} \tr(K_A|_\Delta) \ d\vol. \]

When $\mathcal V$ is $\mathbb Z_2$-graded, we define the $b$-supertrace of $A$ by
\[ \bstr(A) = \intbar_{\widehat M} \str(K_A|_{\Delta}) \ d\vol, \] 
where $\str$ is the fiberwise supertrace on $\textup{End}_{\mathbb Z_2}(\mathcal V)$.

\subsection{The Odd Chern Character in Cyclic Homology}\label{sec:odd}	
For $A$ an algebra over $\mathbb C$, let
\[ C_n(A) = A\otimes (A/\mathbb C)^{\otimes n}. \]
An element of $C_n(A)$ is denoted by $(a_0, a_1, \cdots, a_n )$. Sometimes, we also write $(a_0, a_1, \cdots, a_n )_n$  to emphasise the degree of the element. 
\begin{definition}
	\[ b(a_0, \cdots, a_n) = \sum_{i=0}^{n} (-1)^i (a_0, \cdots, a_ia_{i+1}, \cdots, a_n)\]
	\[ B(a_0, \cdots, a_n) = \sum_{i=0}^{n} (-1)^{ni}(1, a_i, \cdots, a_n, a_0, \cdots, a_{i-1})\]
\end{definition}
Let $C_{+}(A) = \prod_{k } C_{2k}(A)$ and $C_{-}(A) = \prod_k C_{2k+1}(A)$, then we have the chain map
\[b+B : C_{\pm}(A)\to  C_{\mp}(A).\]
The homology of this chain complex is called the periodic cyclic homology of $A$, denoted  $HP_\pm(A)$.

When $A$ is a Banach algebra, we use the inductive tensor product instead of the algebraic tensor product in the definition of $C_n(A) $. We denote the resulted space of continuous even (resp. odd) chains by $C_{+}^{\textup{top}}(A)$ (resp. $C_{-}^{\textup{top}}(A)$). Let us define 
\[ \| c_0 + c_1 + \cdots \|_\lambda = \sup_n \frac{\lambda^n}{\Gamma(n/2)} \|c_n\|. \] 
Then an even chain  $c_0 + c_2 +\cdots \in C_+^{\textup{top}}(A)$ is called entire if $ \| c_0 + c_2 + \cdots \|_\lambda $ is finite for some $\lambda > 0$. Entire odd chains in $C_-^{\textup{top}}(A)$ are defined the same way. 
The space of even (resp. odd) entire chains will be denoted by $C^\omega_{+}(A)$ (resp. $C^\omega_{-}(A)$). It is easy to check that $b$ and $B$ are continuous maps from $C^\omega_{\pm}(A) $ to $ C^\omega_{\mp}(A)$, hence $(C^\omega_{\pm}(A) , b+B )$ is a well defined chain complex. The resulted homology is called the entire cyclic homology of $A$, denoted $H^\omega_\pm(A)$.

For each Banach algebra $A$, the trace map $\Tr: C_n(M_r(A))\to C_n(A)$ by
\[ \Tr(a_0, \cdots, a_n) = \sum_{0 \leq i_0 , \cdots, i_n \leq r} ((a_0)_{i_0i_1}, (a_1)_{i_1 i_2}, \cdots, (a_n)_{i_ni_0}) \]
induces a chain complex homomorphism $C^\omega_{\pm}(M_r(A)) \to C^\omega_{\pm}(A) $. For an invertible element $g\in \textup{GL}_r(A)$, we define its Chern character to be 
\[ \Ch_\bullet(g) := \sum^{\infty}_{k=0} k! \Tr(g^{-1}, g, \cdots, g^{-1}, g)_{2k+1} \]
We have $\Ch_\bullet(g) \in C^\omega_{-}(A)$ and  $(b+B) \Ch_\bullet(g) = 0$.

Similarly, the entire cyclic cohomology of $A$, denoted $H^\pm_\omega$, is defined to be the homology of the cochain complex $(C_\omega^\pm(A), b+B) $, where 
\[ C_\omega^\pm ( A )  := (C^\omega_{\pm}(A))^\ast =   \textup{the space of continuous linear functionals on } \ C^\omega_{\pm}(A),\]
with $b$ and $B$ being the obvious dual maps of those defined for cyclic homology.

\section{JLO Chern Character in \textup{b}-Calculus}\label{sec:JLO}
	
In this section, we shall define the $b$-JLO Chern character and prove its entireness. Let $\widehat M$ be as before and $\mathcal S$ be the spinor bundle over $\widehat M$. Then $\mathcal S_1 = \mathcal S\otimes \mathbb C^{1|1}$ is a Clifford module over $\widehat M$ of degree $1$, where the generator $e_1$ of $\cl_1 (\mathbb C)$ acts on $\mathcal S_1$ by $\begin{pmatrix} 0 & 1 \\ -1 & 0 \end{pmatrix}$. Let $D$ be the Dirac operator on $\widehat M$ and denote 
\[ \twD = i \begin{pmatrix} 0 & D \\ D & 0 \end{pmatrix} \in \bpo^1(\widehat M; \mathcal S_1).\] Notice that $\twD$ is odd and skew-adjoint, and (graded) commutes with the action of $\cl_1(\mathbb C)$.

\subsection{JLO Chern Character in b-Calculus}
For $A\in \bpo^m(\widehat M; \mathcal S_1)$, we define
\[ \bstr_{(1)}(A) :=  \frac{1}{2\sqrt{-\pi} }\bstr(e_1 A).\]
More generally, for $A\in \bpo^m(\widehat M; \mathcal V)$ with $\mathcal V$  a Clifford module of degree $q$, we define
\[ \bstr_{(q)} (A )  := \frac{1}{(-4\pi)^{q/2}}\Str(e_1 \cdots e_q A),\]
where $e_1, \cdots, e_q$ are generators of $\cl_q(\mathbb C)$.

\begin{definition}
The $b$-JLO Chern character of $\twD$ is defined as
	\begin{align*}
	 \bch^n(\twD) (a_0, \cdots, a_n) &:= \bket a_0, [\twD, a_1], \cdots, [\twD, a_n] \rangle\\
	  & = \int_{\Delta^n} \bstr_{(1)} \left(a_0 e^{\sigma_0 \twD^2 } [\twD, a_1] e^{\sigma_1 \twD^2} \cdots  [\twD, a_n] e^{\sigma_n \twD^2} \right) d\sigma
	 \end{align*}
	 where $[ -, - ]$ stands for the graded commutator.
\end{definition}

A straightforward calculation gives the following lemma.
\begin{lemma}\label{lemma:twD}
	\begin{align*} &\bch^{2k+1}(\twD)(a_0, \cdots, a_{2k+1})  \\
		& \quad = \frac{(-1)^k }{\sqrt \pi}\int_{\Delta^{2k+1}} \btr(a_0 e^{-\sigma_0 D^2}[D, a_1] e^{-\sigma_1 D^2} \cdots [D, a_{2k+1}]e^{-\sigma_{2k+1} D^2}) d\sigma
	\end{align*}\qed
\end{lemma}
Recall in the case of closed manifolds, the JLO odd Chern character is defined to be
\begin{align*}
&\Ch^{2k+1}(D)(a_0, \cdots, a_{2k+1}) \\
&= \frac{(-1)^k }{\sqrt \pi}\int_{\Delta^{2k+1}} \Tr(a_0 e^{-\sigma_0 D^2}[D, a_1] e^{-\sigma_1 D^2} \cdots [D, a_{2k+1}]e^{-\sigma_{2k+1} D^2}) d\sigma.
\end{align*}
Hence $\bch^\bullet(\twD)$ is a natural generalization of the JLO odd Chern character to the $b$-calculus setting. 

\begin{lemma}(cf. \cite[Lemma 4.4 ]{EG93}) For $g\in U_r(C^\infty_{\textup{exp}}(\widehat M))$, we have
	\[ \langle \bch^\bullet (\twD) , \sum_{k=0}^{\infty} k!\, \Str(p, \cdots, p )_{2k+1}  \rangle = \langle \bch^\bullet(\twD), \Ch_\bullet (g) - \Ch_\bullet (g^{-1}) \rangle, \]
	where  $ p =  \begin{pmatrix} 0 & -g^{-1} \\ g & 0 \end{pmatrix} \in C^\infty_{\textup{exp}}(\widehat M)\otimes \End(\mathbb C^{r|r})$ with $\mathbb C^{r|r} = (\mathbb C^r)^+\oplus (\mathbb C^r)^-$ being $\mathbb Z_2$-graded.  
\end{lemma}
\begin{proof}

Notice that 
\[ [\twD, p] = \begin{pmatrix} 0 & -[\twD, g^{-1}] \\ [\twD, g] & 0 \end{pmatrix}  \]
and 
\begin{align*}  
&\quad [\twD, p]e^{-\sigma \twD^2} [\twD, p] e^{-\tau \twD^2} \\
 &= \begin{pmatrix} [\twD, g^{-1}] e^{-\sigma \twD^2}[\twD, g] e^{-\tau \twD^2} & 0 \\ 0 & [\twD, g] e^{-\sigma \twD^2}[\twD, g^{-1}] e^{-\tau \twD^2} \end{pmatrix}.  
\end{align*}
Then
\begin{align*}
	&\bket p, [\twD, p], \cdots, [\twD, p] \rangle \\
	& = \bket g^{-1}, [\twD, g], \cdots, [\twD, g^{-1}], [\twD, g] \rangle -\bket g, [\twD, g^{-1}], \cdots, [\twD, g], [\twD, g^{-1}] \rangle	
\end{align*}
Hence follows the lemma.

\end{proof}

\subsection{Entireness of the \textup{b}-JLO Chern Character}\label{subsec:bound}

For $A\in \bpo^{-\infty}(\widehat M, \mathcal V)$, we denote
\[  \Tr^M(A) := \int_M \tr(K_A|_{\Delta}) \quad\textup{and} \quad \btr^{\textup{end}}(A) := \intbar_{(-\infty, 0]\times \partial M}\tr(K_A|_\Delta). \]
When $\mathcal V$ is a Clifford module of degree $1$,  we define
\[\Str_{(1)}^{M}\left(A \right) := \int_M\str_{(1)} (K_A|_\Delta)\quad \textup{and} \quad \bstr_{(1)}^{\textup{end}}(A) := \intbar_{(-\infty, 0]\times \partial M} \str_{(1)}(K_A|_\Delta). \]
When $A|_{(-\infty, 0]\times \partial M}$ is of trace class, we also write $ \Str_{(1)}^{\textup{end}} (A) $ instead of $\bstr_{(1)}^{\textup{end}}(A)$. 
Now let us give an upper bound in terms of $\bnorm a_i\|$ for   	
\begin{align}&\int_{\Delta^n} \bstr_{(1)} \left(a_0 e^{\sigma_0 \twD^2 } [\twD, a_1] e^{\sigma_1 \twD^2} \cdots  [\twD, a_n] e^{\sigma_n \twD^2} \right) d\sigma \\
&= \int_{\Delta^n} \Str_{(1)}^M \left(a_0 e^{\sigma_0 \twD^2 } [\twD, a_1] e^{\sigma_1 \twD^2} \cdots  [\twD, a_n] e^{\sigma_n \twD^2} \right) d\sigma  \label{eq:sumone}\\
&\quad + \int_{\Delta^n} \bstr_{(1)}^{\textup{end}} \left(a_0 e^{\sigma_0 \twD^2 } [\twD, a_1] e^{\sigma_1 \twD^2} \cdots  [\twD, a_n] e^{\sigma_n \twD^2} \right) d\sigma \label{eq:sumtwo}
\end{align}
For the first summand $\eqref{eq:sumone}$, by standard differential calculus on compact manifolds, one has 
\[ \left| \int_{\Delta^n} \Str_{(1)}^M \left(a_0 e^{\sigma_0 \twD^2 }\cdots [\twD, a_n] e^{\sigma_n \twD^2} \right) d\sigma \right| \leq \Tr^M(e^{\twD^2})\, \bnorm a_0\| \bnorm a_1\| \cdots \bnorm a_n\|
 \]
cf. \cite[Lemma 2.1]{EG-AS89}.

For the second summand $\eqref{eq:sumtwo}$, first notice that on $(-\infty, 0]\times \partial M$, 
\begin{align*}
 [\twD, a] &= c(d a_c) + e^x \left[ c(a_\infty dx) + c(da_\infty)\right]\\
   &= C + e^x B
\end{align*}
where $C = c(d a_c) $ and $B = \left[ c(a_\infty dx) + c(da_\infty)\right]$ with  $c(-)$ denoting the Clifford multiplication. Similarly, we write  
\[ [\twD, a_i] = C_i + e^x B_i \quad  \textup{and} \quad a_0 = C_0 + e^x B_0 , \]
where $C_i$ is constant along the normal direction $x$. Notice that $\|B_i\| \leq \bnorm a_i\|$ and $\|C_i\|\leq \bnorm a_i\|$. The term $\eqref{eq:sumtwo}$ can be written as a sum of terms of the following two types:
\begin{enumerate}[(I)]
	\item $\int_{\Delta^n} \bstr_{(1)}^{\textup{end}} \left(C_0 e^{\sigma_0 \twD^2 } \cdots C_n e^{\sigma_n \twD^2} \right) d\sigma$, \\ 
\item  $ \int_{\Delta^n} \bstr_{(1)}^{\textup{end}} \left(C_0 e^{\sigma_0 \twD^2 } \cdots e^{\sigma_i \twD^2 } e^t B_i\, e^{\sigma_{i+1} \twD^2 } \cdots C_n e^{\sigma_n \twD^2} \right) d\sigma$ .
\end{enumerate}
Let us denote the Dirac operator $\mathbb R \times \partial M$ by $D_{\mathbb R}$ and write $\twD_{\mathbb R} = i\begin{pmatrix} 0 & D_{\mathbb R}  \\ D_{\mathbb R} & 0 \end{pmatrix} $. By \cite[Proposition 3.1]{L-M-P09}, $(e^{\sigma\twD_{\mathbb R}^2} - e^{\sigma\twD^2})|_{(-\infty, 0]\times \partial M}$ is of trace class and there is a constant $\mathcal K_0$ such that  
\begin{equation}\label{eq:switch}
 \left|\Tr(e^{\sigma\twD_{\mathbb R}^2} - e^{\sigma\twD^2})|_{(-\infty, 0]\times \partial M}\right| \leq \mathcal K_0  \quad \textup{for all $ 0\leq \sigma \leq 1$ }.
\end{equation}
\textbf{Type I. } Since $\|e^{-\sigma\twD^2_{\mathbb R}} \| \leq 1 $ and  $\|e^{-\sigma\twD^2} \| \leq 1 $, one has
\begin{align*}
	&\left|\bstr_{(1)}^{\textup{end}} \left(C_0 e^{\sigma_0 \twD^2 } \cdots C_n e^{\sigma_n \twD^2}|_{(-\infty, 0]\times \partial M} \right) \right| \\
	&  \leq (n+1) \mathcal K_0 \prod_{i=0}^{n} \|C_i\| + \left|\bstr_{(1)}^{\textup{end}} \left(C_0 e^{\sigma_0 \twD^2_{\mathbb R} } \cdots C_n e^{\sigma_n \twD^2_{\mathbb R}}|_{(-\infty, 0]\times \partial M} \right) \right| \\
	& = (n+1) \mathcal K_0 \prod_{i=0}^{n} \|C_i\|
\end{align*}
where the last equality follows from the fact 
$$\bstr_{(1)}^{\textup{end}} \left(C_0 e^{\sigma_0 \twD^2_{\mathbb R} } \cdots C_n e^{\sigma_n \twD^2_{\mathbb R}}|_{(-\infty, 0]\times \partial M} \right) = 0$$
 by the definition of the $b$-trace.
 
\noindent\textbf{Type II.} 
Due to  the presence of the factor $e^x$, 
\[  C_0 e^{\sigma_0 \twD^2 } \cdots e^{\sigma_i \twD^2 } e^t B_i\, e^{\sigma_{i+1} \twD^2 } \cdots C_n e^{\sigma_n \twD^2} \]
is of trace class. 

Without loss of generality, it suffices to give an upper bound for 
\begin{align*}
 &\Str_{(1)}^{\textup{end}} \left(e^x B_0 e^{\sigma_0 \twD^2 } A_1 e^{\sigma_1 \twD^2} \cdots A_n e^{\sigma_n \twD^2} \right) 
\end{align*}
where $A_i = B_i$ or $C_i$ as defined above for $1 \leq i\leq n$. First by the inequality $\eqref{eq:switch}$,
\begin{align*}
	&\left|\Str_{(1)}^{\textup{end}} \left(e^x B_0 e^{\sigma_0 \twD^2 } A_1 e^{\sigma_1 \twD^2} \cdots A_n e^{\sigma_n \twD^2}\right) \right| \\
	&  \leq (n+1) \mathcal K_0 \| B_0\| \prod_{i=1}^{n} \|A_i\| + \left|\Str_{(1)}^{\textup{end}} \left(e^x B_0 e^{\sigma_0 \twD_{\mathbb R}^2 } A_1 e^{\sigma_1 \twD_{\mathbb R}^2} \cdots A_n e^{\sigma_n \twD^2_{\mathbb R}}|_{(-\infty, 0]\times \partial M} \right) \right|
\end{align*}
Now we can rewrite  
\begin{align*}
& e^x B_0 e^{\sigma_0 \twD_{\mathbb R}^2 } A_1 e^{\sigma_1 \twD_{\mathbb R}^2} \cdots A_n e^{\sigma_n \twD_{\mathbb R}^2}\\
&= (e^x B_0 e^{\sigma_0 \twD_{\mathbb R}^2 } e^{-\beta_1 x}) (e^{\beta_1 x} A_1 e^{\sigma_1 \twD_{\mathbb R}^2} e^{-\beta_2 x})  e^{\beta_2 x} \cdots  e^{-\beta_n x} (e^{\beta_n x} A_n e^{\sigma_n \twD_{\mathbb R}^2}) 
\end{align*}
with $1 > \beta_1 >\beta_2 >\cdots > \beta_n >0 $. By \cite[Proposition 3.7]{L-M-P09}, there is a constant ${\mathcal K'}$ such that 
\[ \|e^{\beta_i x} A_i e^{\sigma_i \twD_{\mathbb R}^2} e^{-\beta_{i+1}x}\|_{\sigma_i^{-1}} \leq \mathcal K'\|A_i\| (\beta_{i} -\beta_{i+1})^{-\sigma_i} \left(\sigma_i^{-\frac{\dim M + 1}{2} \sigma_i}\right) \]
for all $i$. 
Notice that 
\[ x^{-\frac{\dim M + 1}{2} x} = e^{-\frac{\dim M + 1}{2} x \ln(x)} \]
is bounded on $[0, 1]$. If we take $\beta_i = (n+1-i)/(n+1)$, then by H\"{o}lder inequality one has
\begin{align*}
 & \left|\Str_{(1)}^{\textup{end}} \left(e^x B_0 e^{\sigma_0 \twD_{\mathbb R}^2 } A_1 e^{\sigma_1 \twD_{\mathbb R}^2} \cdots A_n e^{\sigma_n \twD^2_{\mathbb R}}|_{(-\infty, 0]\times \partial M} \right) \right|\\
 & \leq \mathcal K_1^{n+1} (n+1) \| B_0\| \prod_{i=1}^{n} \|A_i\|
\end{align*}
for some fixed constant $\mathcal K_1$.

Applying the estimates above, we have the following proposition.
\begin{proposition}\label{prop:bchentire}
$\bch^\bullet (\twD)$ is an entire cyclic cochain.
\end{proposition}
\begin{proof}	
\begin{align}
 &\quad  |\bch^n(\twD) (a_0, \cdots, a_n)| \notag\\
 & = \left| \bket a_0, [\twD, a_1], \cdots, [\twD, a_n] \rangle \right|\notag\\
	& = \left| \int_{\Delta^n} \bstr_{(1)} \left(a_0 e^{\sigma_0 \twD^2 } [\twD, a_1]e^{\sigma_1 \twD^2} \cdots [\twD, a_n] e^{\sigma_n \twD^2} \right) d\sigma \right| \notag\\ 
	& \leq \frac{2^n (n+1)\left(\mathcal K_1^n +  2\mathcal K_0\right)}{n!} \bnorm a_0\| \bnorm a_1\| \cdots \bnorm a_n\|\label{eq:bound}
\end{align}	
It follows that $\bch^\bullet (\twD)$  defines a continuous linear functional on $C^\omega_-(A)$, i.e. an entire cyclic cochain in $C^{-}_\omega(A)$
\end{proof}

\section{Odd APS Index Theorem for manifolds with boundary}\label{sec:main}
In this section, we shall state and prove the main theorem of this paper. Let $\widehat M$ be an odd dimensional spin $b$-manifold with a $b$-metric as before and $D$ its associated Dirac operator. Recall that 
\[ \twD = i \begin{pmatrix} 0 & D \\ D & 0 \end{pmatrix}  \]
We put $\twD_t = t \twD$ and define
\[ \bch^k ( \twD , t) (a_0, a_1, \cdots, a_k) := \ldb a_0, [ \twD_t, a_1], \cdots, [\twD_t, a_k] \rdb \]
where 
\[ \ldb A_0 , A_1, \cdots, A_k \rdb := \int_{\Delta^k} \bstr_{(1)}(A_0 e^{\sigma_0(d\twD_t + \twD_t^2)} \cdots A_k e^{\sigma_k (d\twD_t +\twD_t^2)}) d\sigma,  \] 
with
	\[e^{s(d\twD_t + \twD_t^2)} := \sum_{k=0}^{\infty} \int_{\Delta^k} e^{\sigma_0 s \twD_t^2} d\twD_t e^{\sigma_1 s\twD_t^2} \cdots d\twD_t e^{\sigma_k s \twD_t^2} d\sigma. \]
Notice that  ( cf. \cite[Lemma 2.5]{EG93b})
	\[ \ldb A_0 , A_1, \cdots, A_k \rdb = \bket A_0, \cdots, A_k\rangle + \sum_{i=0}^k \bket A_0, \cdots, A_i, dt\twD, A_{i+1}, \cdots, A_k \rangle\]
Therefore
\begin{align}
	&\bch^k(\twD, t) (a_0, \cdots, a_k) = \bket a_0, [\twD_t, a_1] , \cdots, [\twD_t, a_k] \rangle \notag\\
	   &\qquad + \sum_{i=0}^{k} (-1)^i dt \bket a_0, [\twD_t, a_1], \cdots, [\twD_t, a_i], \twD, [\twD_t, a_{i+1}], \cdots, [\twD_t, a_k] \rangle.\label{eq:oneform}
\end{align}
Since $\btr$ is not a trace, $\bch^\bullet (\twD, t)$ is not a closed cochain. However one has (\cite[Theorem 6.2]{EG93b})
\begin{equation}\label{eq:getzler}
(d+ b+ B)\bch^\bullet (\twD, t) = \Ch^\bullet( \bp \twD, t) 
\end{equation}
where $ \bp \twD$ is the restriction of $\twD$ to the boundary. 
Let $\alpha \in \Omega^\ast(0, \infty)$ be the differential form $ \alpha = \langle \bch^\bullet (\twD, t) , \Ch_\bullet(g) \rangle $, then 
\[  d\alpha = \langle \Ch^\bullet( \bp \twD, t) , \Ch_\bullet(g_\partial) \rangle. \]

By the fundamental theorem of calculus,  
\[ \alpha(t_2) - \alpha(t_1) = \langle \int_{t_1}^{t_2} \Ch^\bullet( \bp \twD, t) , \Ch_\bullet (\bp g) \rangle  \]
which implies that 
\[ \lim_{t\to \infty}\alpha(t) - \lim_{t\to 0}\alpha(t) = \left\langle \eta^\bullet( \bp \twD) , \Ch_\bullet (\bp g) \right\rangle, \]
provided that both limits exist.
Here $ \eta^\bullet( \bp \twD) := \int_{0}^{\infty} \Ch^\bullet( \bp\twD, t) = 2 \,\eta^\bullet(\bp D)	$, where $ \eta^\bullet (\bp { D}):= \int_0^\infty \Ch^\bullet (\bp { D}, t) $ is the higher eta cochain of Wu \cite{FW93}. More explicitly, 
\begin{align*} 
&\eta^{2k+1} (\bp D) (a_0, a_1, \cdots, a_{2k+1}) \\
& =  (-1)^k\sum_{l=0}^{k} (-1)^l \int_0^\infty dt \langle a_0, [\bp { D}_t, a_1], \cdots, [\bp { D}_t, a_i], \bp { D}, [\bp{ D}_t, a_{i+1}], \cdots, [\bp{ D}_t, a_k] \rangle
\end{align*}
where $\langle a_0, \cdots, a_k\rangle = \int_{\Delta^k} \Tr(a_0 e^{\sigma_0 \bp D_t^2} \cdots a_k e^{\sigma_k \bp D_t^2}) d\sigma$. The higher eta cochain has a finite radius of convergence, cf. \cite[Proposition 1.5]{FW93}. In order for $$\langle \eta^\bullet( \bp D) , \Ch_\bullet (\bp g)\rangle$$ to converge, we shall make the following assumptions throughout the rest of this paper:
\begin{enumerate}
			\item \textit{$\bp D$ is invertible and the lowest eigenvalue of $|\bp D|$ is $\lambda$};
			\item $\|[\bp D, g]\| < \lambda$.
\end{enumerate}
		
Let us denote by
\[ \Ch^{\textup{dR}}_\bullet(g) := \sum_{k=0}^\infty \frac{1}{(2\pi i)^{k+1}}\frac{k!}{(2k+1)!} \tr\left( (g^{-1}dg)^{2k+1}\right) \in \Omega^\ast (\widehat M) \]
the Chern character of $g$ in de Rham cohomology of $\widehat M$. Then we have the following main theorem of this paper. 
\begin{theorem}\label{thm:main}
Let $\widehat M$ be an odd dimensional spin $b$-manifold with a $b$-metric and $D$ its associated Dirac operator. Assume $\bp D$ is invertible. For $g\in U_k(C^\infty(\widehat M))$ a unitary over $\widehat M$, if $\|[\bp D, g]\| < \lambda$ where $\lambda$ is the lowest nonzero eigenvalue of $|\bp D|$, then
\begin{equation*}
\sflow ( D, g^{-1} D g) = \ \intbar_{\widehat M} \ahat(\widehat M) \wedge \Ch_\bullet^{\textup{dR}}(g) + \left\langle \eta^\bullet(\bp D) , \Ch_\bullet (\bp g) \right\rangle.
\end{equation*}
where $\ahat (\widehat M) = \det \left( \frac{\bnab^2/4\pi i}{\sinh \bnab^2/4\pi i}\right)^{1/2}$ with $\bnab$ the Levi-Civita $b$-connection associated to the $b$-metric on $\widehat M$.
\end{theorem}
\begin{proof}
We need to identify the limits of $\alpha(t) $ for $t =\infty $ and $t=0$. In the case of closed manifolds, the local formula for the limit of $\alpha(t)$ as $t\to 0$ follows from Getzler's asymptotic calculus, cf. \cite{JBJF90,CM90,EG83}. A direct calculation in the $b$-calculus setting is carried out in \cite[Section 5 $\&$ 6]{L-M-P09}. In particular, one has 
\begin{align*}
\lim_{t\to 0}\alpha(t) &=  \lim_{t\to 0} \langle \bch^\bullet(\twD), \Ch_\bullet (g) - \Ch_\bullet (g^{-1}) \rangle \\
& =\intbar_{\widehat M}  \ahat(\widehat M) \wedge \left( \Ch^{\textup{dR}}_\bullet(g) -  \Ch^{\textup{dR}}_\bullet(g^{-1}) \right)\\
&= 2 \intbar_{\widehat M}  \ahat(\widehat M) \wedge  \Ch^{\textup{dR}}_\bullet(g)
\end{align*}
where we have used the fact that $\Ch^{\textup{dR}}_\bullet(g^{-1}) = - \Ch^{\textup{dR}}_\bullet(g)$.
 
Now the theorem follows once we show 
\[\lim_{t\to \infty}\alpha(t) = 2\, \sflow(D, g^{-1}Dg), \]
which will be proved in Proposition $\ref{prop:large}$ below.
\end{proof}
		
\begin{corollary} With the same notation as the above theorem, 
\[ \left\langle \eta^\bullet(\bp D) , \Ch_\bullet (\bp g) \right\rangle = \eta(\partial \widehat M, \bp g) \mod \mathbb Z \]
\end{corollary}	
\begin{proof}
Here $\eta(\partial \widehat M, \bp g)$ is the eta invariant of Dai and Zhang \cite{XD-WZ06}. Without loss of generality, we can assume the unitary $g$ is constant along the the normal direction of the cylindrical end. In this case, we have   
\[ \intbar_{\widehat M}  \ahat(\widehat M) \wedge  \Ch^{\textup{dR}}_\bullet(g) = \int_{ M}  \ahat(M) \wedge  \Ch^{\textup{dR}}_\bullet(g)\] 
by definition of the regularized integral. Now comparing the above theorem with the Toeplitz index theorem on odd dimensional manifolds by Dai and Zhang \cite[Theorem 2.3]{XD-WZ06}, one has 
\[ \left\langle \eta^\bullet(\bp D) , \Ch_\bullet (\bp g) \right\rangle = \eta(\partial \widehat M, \bp g) \mod \mathbb Z .\]

\end{proof}
This equality provides more evidence for the naturality of the Dai-Zhang eta invariant for even dimensional closed manifolds.

\section{Spectral Flow}\label{sec:sflow}
In this section, we proceed to explain the notion of spectral flow and prove an analogue of Getzler's formula for spectral flow (cf. \cite[Corollary 2.7]{EG93}) in the $b$-calculus setting.

Following Booss-Bavnbek, Lesch and Phillips \cite{BLP05}, we define the notion of spectral flow as follows.
\begin{definition} 
Let $T_u: \mathcal H\to \mathcal H$ for $u\in [0, 1]$ be a continuous path of (possibly unbounded) self-adjoint Fredholm operators, then its spectral flow, denoted by $\sflow(T_u)_{0\leq u \leq 1}$, is defined by
\[  \sflow(T_u)_{0\leq u \leq 1} = \textup{wind} \left(\kappa(T_u)_{0\leq u \leq 1}\right) \]
where $\kappa(T) = (T-i)(T+i)^{-1}$ is the Cayley transform of $T$ and $\textup{wind} \left(\kappa(T_u)_{0\leq u \leq 1}\right) $ is the winding number of the path $\kappa(T_u)_{0\leq u \leq 1}$ (see also \cite[Section 6]{PK-ML04}). We also write $\sflow(T_0, T_1)$ for the spectral flow if it is clear what the path is from the context.  
\end{definition}

Actually, in this paper where we are concerned with smooth paths of self-adjoint Fredholm operators,  we use the following equivalent working definition of the spectral flow (cf. \cite[Section 2.2]{BLP05}). Let $T_u: \mathcal H\to \mathcal H$ for $u\in [0, 1]$ be a smooth path of (possibly unbounded) self-adjoint Fredholm operators. For a fixed $u_0\in [0, 1]$, there exists $(a, b) \subset [0, 1]$ such that
\begin{enumerate}
\item $u_0 \in (a, b)$ (unless $u_0 = 0$ or $1$, in which case $u_0 = a = 0$ if $u_0=0$ and $u_0 = b = 1$ if $u_0 =1$);
\item $\dim\ker(T_u) \leq \dim\ker(T_{u_0})$ for all $u\in (a, b)$.
\end{enumerate}
 By shrinking the neighborhood $(a, b)$ if necessary, we can assume that the essential spectrum of $|T_u|$ for $u\in (a, b)$ is bounded below uniformly by $\lambda_{0}$ and the spectrum of $T_u$ in $(-\lambda_0,  \lambda_0)$ consists of discrete eigenvalues. We can further assume $T_u$ has the same number of eigenvalues (counted with multiplicities) in $(-\lambda_0, \lambda_0)$, for all $u\in (a, b)$. By perturbation theory of linear operators (cf. \cite[II.6, V.4.3, VII.3 ]{TK95}), there are smooth functions $\beta_k$ on $(a, b)$ such that $\{ \beta_k(u)  \}_k$ gives a complete set of eigenvalue of $T_u$ in $(-\lambda_0, \lambda_0)$. Let $n_b $ (resp. $n_a$)  be the number of nonnegative eigenvalues of $T_b$ (resp. $T_a$ ) in $(-\lambda_0, \lambda)$. Then we define the spectral flow of $(T_u)_{a\leq u \leq b}$ to be
  \begin{equation}\label{eq:sf}
   \sflow (T_u)_{a\leq u \leq b} := (n_b - n_a).  
  \end{equation} 
 We call an interval $(a, b)$ as above together with $u_0 \in (a, b)$  a pointed gap interval. It is easy to see that the formula $\eqref{eq:sf}$ is additive with respect to disjoint pointed gap intervals. Let us cover $[0, 1]$ by finitely many intervals, say $[a_i, b_i]_{0\leq i \leq n}$ such that each $(a_i, b_i)$ is a pointed gap interval, with $b_i = a_{i+1}$,  $u_0 = a_0 = 0$, $u_n = b_n =1$ and $u_j \in (a_j, b_j)$ for $1 \leq j \leq n-1$.  Then we define 
\begin{equation} \sflow (T_u)_{0\leq u\leq 1} := \sum_{ j = 0 }^{n}  \sflow (T_u)_{a_j \leq u \leq b_j}.
\end{equation}
By additivity of formula $\eqref{eq:sf}$, we see that $\sflow(T_u)_{0\leq u \leq 1}$ defined as such is independent of the choice of pointed gap intervals.

Let $\widehat M$ be an odd dimensional spin $b$-manifold with a $b$-metric as before and $D$ its associated Dirac operator. Let $D_u = (1-u) D + u g^{-1} D g$. Since $\|[\bp D, g]\| < \lambda$ by assumption, $\bp D_u$ is invertible for all $u\in [0, 1]$. It follows that $\inf\textup{spec}_{\textup{ess}}(|D_u|) > 0$ for all $u\in [0, 1]$. Thus $\{D_u\}_{0\leq u\leq 1}$ is an analytic family of self-adjoint Fredholm operators.

Following from the discussion above, we see that for each fixed $u_0\in [0, 1]$, there exists $(a, b) \subset [0, 1]$ and $\lambda_0 >0$ such that the spectrum of $D_u$  in $(-\lambda_0, \lambda_0)$ consists of discrete eigenvalues for all $u\in (a, b)$. Moreover, we can assume $D_u$ has the same number of eigenvalues in $(-\lambda_0, \lambda_0)$, for all $u\in (a, b)$. We put
	\begin{align}
	 A_u &:= D_u P_u, \\
	B_u &:= D_u (I - P_u) + P_u,\\ 
  C_u &:= D_u (I - P_u), \label{eq:cu}
\end{align}	
where $P_u$ is the spectral projection of $D_u$ on $(-\lambda_0, \lambda_0)$. Let $\beta_k$ be the smooth functions on $(a, b)$ such that $\{ \beta_k(u)  \}_k$ gives the complete set of eigenvalues  (counted with multiplicities) of $A_u$. Since $\{D_u\}_{0\leq u \leq 1}$ is an analytic family of operators, $\beta_k$ is an analytic function of $u\in (a, b)$. It follows that for each $k$, $\beta_k$ either has only finitely many isolated zeroes or is itself constantly zero. Hence by shrinking $(a, b)$ as much as needed, we can assume $\beta_k$ either is a constant zero function or has only one zero in $(a, b)$. In the latter case, by shrinking $(a, b)$ again if necessary, we can assume the isolated zeros can only happen at $u_0$. Moreover, for each $u\in (a, b)$, there is a set of orthonormal eigenvectors $\{\phi_k(u)\}_{1\leq k\leq m}$ such that $A_u \phi_k(u) = \beta_k(u) \phi_k(u) $ and the vector-valued function $\phi_k  $ is analytic with respect to $u$ for each $1 \leq k\leq m$. 

Following Getzler \cite{EG93}, we define the truncated eta invariant of $D$ to be
\[ \eta_\varepsilon (D) := \frac{1}{\sqrt \pi} \int_\varepsilon^\infty \btr(De^{-sD^2}) s^{-1/2}ds =  \frac{2}{\sqrt \pi} \int_\varepsilon^\infty \btr(De^{-t^2D^2}) dt. \]
and the reduced (truncated) eta invariant of $D$ to be
\[ \reta_\varepsilon(D) = \frac{\eta_\varepsilon(D) + \dim\ker(D)}{2}. \]
The following lemma is a natural extension of \cite[Proposition 2.5]{EG93} to the $b$-calculus setting.
\begin{lemma}\label{lemma:vareta}
	\[\frac{d \eta_\varepsilon (B_u)}{du} = -\frac{2\varepsilon}{\sqrt \pi} \btr(\dot{B}_u  e^{-\varepsilon^2 B_u^2}) +E_\varepsilon(u) \]
	where $E_\varepsilon(u)$ is defined by
	\begin{align*}
		 E_\varepsilon(u) & = -\frac{2}{\sqrt \pi}\int_\varepsilon^\infty \int_0^1 t^2 \ \btr\left[e^{-st^2 B_u^2} B_u^2 \ , \ \dot{B}_u e^{-(1-s)t^2 B_u^2}\right] ds dt \\
		& \quad -\frac{2}{\sqrt \pi}\int_\varepsilon^\infty \int_0^1 t^2 \ \btr\left[e^{-st^2 B_u^2} B_u \ , \ \dot{B}_u B_u e^{-(1-s)t^2 B_u^2}\right] ds dt. 
	\end{align*}
\end{lemma}
\begin{proof}
	Using Duhamel's principle, we have 
\begin{align*}
\frac{d }{du}\eta_\varepsilon (B_u) &= 	\frac{2}{\sqrt \pi} \int_{\varepsilon}^{\infty} \btr(\dot{B}_u e^{-t^2 B_u^2})dt  \\
       							& \quad -\frac{2}{\sqrt \pi} \int_\varepsilon^\infty \int_0^1 \btr\left(B_u e^{-st^2 B_u^2}t^2(B_u \dot{B}_u + \dot{B}_u B_u )e^{-(1-s)t^2 B_u^2}\right)ds dt \\
& =\frac{2}{\sqrt \pi} \int_{\varepsilon}^{\infty} \btr(\dot{B}_u e^{-t^2 B_u^2})dt  - \frac{4}{\sqrt \pi}\int_{\varepsilon}^{\infty} \btr(t^2\dot{B}_u B_u^2 e^{-t^2 B_u^2})dt + E_\varepsilon (u),
\end{align*}

Integration by parts shows that 
\begin{align*}
	\int_{\varepsilon}^{\infty} \btr(t^2\dot{B}_u B_u^2 e^{-t^2 B_u^2})dt &= -\frac{1}{2}\int_{\varepsilon}^{\infty} t \frac{d}{dt}\btr(\dot{B}_u  e^{-t^2 B_u^2}) dt \\
	& = -\frac{1}{2}\left.  \btr(\dot{B}_u  e^{-t^2 B_u^2}) t\right|_{t=\epsilon}^{t=\infty} + \frac{1}{2}\int_{\varepsilon}^{\infty} \btr(\dot{B}_u  e^{-t^2 B_u^2}) dt 
\end{align*}
Since $B_u$ is invertible, $\btr(t\dot{B}_u  e^{-t^2 B_u^2}) $ goes to $0$ as $t\to \infty$.
It follows that 
\[ \frac{d }{du}\eta_\varepsilon (B_u) = -\frac{2\varepsilon}{\sqrt \pi} \btr(\dot{B}_u  e^{-\varepsilon^2 B_u^2}) + E_\varepsilon(u). \]

\end{proof}

\begin{corollary}
	For $u \in (a, b)$,
	\[\frac{d \eta_\varepsilon (C_u)}{du} = -\frac{2\varepsilon}{\sqrt \pi} \btr(\dot{C}_u  e^{-\varepsilon^2 C_u^2}) +E_\varepsilon (u)\]
\end{corollary}
\begin{proof}
By definition, we have 
\[ \eta_\varepsilon (C_u) = \eta_\varepsilon (B_u) - K\int_\varepsilon^\infty e^{-t^2}dt ,\]
where $K = \textup{rank}(P_u)$ is independent of $u\in (a, b)$.
Thus $\frac{d}{du} \eta_\varepsilon (C_u) = \frac{d}{du} \eta_\varepsilon (B_u).$
Notice that 
\begin{align*} 
& \btr(\dot{B}_u  e^{-\varepsilon^2 B_u^2})  \\
&=  \btr(\dot{C}_u  e^{-\varepsilon^2 C_u^2}) + \btr(\dot{P}_u  e^{-\varepsilon^2 P_u^2}) + \btr(\dot{C}_u  e^{-\varepsilon^2 P_u^2}) + \btr(\dot{P}_u  e^{-\varepsilon^2 C_u^2}).\\
 &=  \btr(\dot{C}_u  e^{-\varepsilon^2 C_u^2}) + \Tr(\dot{P}_u  e^{-\varepsilon^2 P_u^2})  + \Tr(\dot{C}_u  e^{-\varepsilon^2 P_u^2}) + \Tr(\dot{P}_u  e^{-\varepsilon^2 C_u^2}).
\end{align*}
since $\dot{P}_u  e^{-\varepsilon^2 P_u^2}, \dot{C}_u  e^{-\varepsilon^2 P_u^2}$ and $\dot{P}_u  e^{-\varepsilon^2 C_u^2}$ are all trace class operators. In fact, since $P_u$ is a projection and the rank of $P_u$ remains constant for each $u\in (a, b)$, using Duhamel's formula, we have
\[ \Tr(\dot{P}_u  e^{-\varepsilon^2 P_u^2}) =  \frac{1}{2}\Tr\left((\dot{P}_u P_u + P_u \dot{P}_u)  e^{-\varepsilon^2 P_u^2}\right ) = \frac{-1}{2\varepsilon^2}\frac{d}{du}\Tr(e^{-\varepsilon^2 P_u^2}) = 0.\]
By the very definition of $C_u$ (see Formula $\eqref{eq:cu}$ above), we have $P_u C_u = C_u P_u = 0$. In particular, $\Tr(C_u  e^{-\varepsilon^2 P_u^2}) \equiv 0$.  Therefore, 
\begin{align*}
0 = \frac{d}{du}\Tr(C_u  e^{-\varepsilon^2 P_u^2}) &= \Tr(\dot{C}_u  e^{-\varepsilon^2 P_u^2})  + \Tr\left(C_u e^{-\varepsilon^2 P_u^2} (P_u\dot{P}_u + \dot{P}_u P_u)\right) \\ 
& = \Tr(\dot{C}_u  e^{-\varepsilon^2 P_u^2})  
\end{align*}
Similarly,  $\Tr(\dot{P}_u  e^{-\varepsilon^2 C_u^2}) = 0$. We conclude that 
\[  \btr(\dot{B}_u  e^{-\varepsilon^2 B_u^2}) = \btr(\dot{C}_u  e^{-\varepsilon^2 C_u^2}).\]
Hence follows the corollary.
\end{proof}

\begin{lemma} For $	\tau \in (a, b)$ and $\tau \neq u_0$, we have 
\[ \left.\frac{d}{du}\eta_\varepsilon(A_u)\right|_{u = \tau} =  -\left.\frac{2\varepsilon}{\sqrt \pi} \Tr( \dot{A}_u  e^{-\varepsilon^2 A_u^2})\right|_{u = \tau}.  \]
\end{lemma}
\begin{proof}Notice that 
\begin{align*}
	\eta_\varepsilon(A_u) & = \sum_{k}\frac{2}{\sqrt \pi}\int_\varepsilon^\infty \beta_k(u) e^{-t^2\beta_k^2(u)} dt \\
	& = \sum_{k}\frac{2}{\sqrt \pi} \int_\varepsilon^\infty \langle A_u e^{-t^2A_u^2}\phi_k(u), \ \phi_k(u) \rangle dt.
\end{align*}
If $\beta_k(\tau)\neq 0$, then the same argument from Lemma $\ref{lemma:vareta}$ shows that 
\[ \left.\frac{d}{du}\int_\varepsilon^\infty \langle A_u e^{-t^2A_u^2}\phi_k(u), \ \phi_k(u) \rangle dt\right|_{u = \tau} = -\left.\varepsilon\langle \dot{A}_u  e^{-\varepsilon^2 A_u^2} \phi_k(u), \phi_k(u) \rangle\right|_{u = \tau}. \]
If $\beta_k(\tau) = 0$,  then $\beta_k \equiv 0 $ on $(a, b)$ by our choice of the interval $(a, b)$. In particular, $A_u\phi_k(u) = 0$ for all $u\in (a, b)$. 
Then 
\begin{align*} 
\frac{d}{du}\langle A_u\phi_k(u), \phi_k(u) \rangle &= \langle \dot{A}_u\phi_k(u), \phi_k(u) \rangle +\langle A_u\dot\phi_k(u), \phi_k(u) \rangle + \langle A_u\phi_k(u), \dot\phi_k(u) \rangle \\
      &= \langle \dot{A}_u\phi_k(u), \phi_k(u) \rangle.
\end{align*}	 
It follows that $\langle \dot{A}_u\phi_k(u), \phi_k(u) \rangle = 0$. Hence
\[ \langle \dot{A}_u  e^{-\varepsilon^2 A_u^2} \phi_k(u), \phi_k(u) \rangle  =  \langle \dot{A}_u   \phi_k(u), \phi_k(u) \rangle = 0 \]
for all $u \in (a, b)$. This finishes the proof.
\end{proof}

\begin{corollary}
	For $u \in (a, b)$ and $u\neq u_0$,
	\[ \frac{d}{du}\eta_\varepsilon(D_u) du =  -\frac{2\varepsilon}{\sqrt \pi} \btr( \dot{D}_u  e^{-\varepsilon^2 D_u^2}) + E_\varepsilon (u).  \]
\end{corollary}
\begin{proof}
Notice that 
\[ \eta_\varepsilon (D_u) = \eta_\varepsilon (A_u ) + \eta_\varepsilon(C_u)\]
and
\[ \btr(\dot{A}_u  e^{-\varepsilon^2 C_u^2}) = \btr(\dot{C}_u  e^{-\varepsilon^2 A_u^2})  = 0. \]
The corollary follows from the above lemmas.
\end{proof}
If we denote by $Q^+ $ the cardinality of the set $\{ \beta_k \mid \beta_k(u_0) = 0 \ \textup{and} \ \beta_k(a) >0 \}$ and $Q^- $ be the cardinality of the set $\{ \beta_k \mid \beta_k(u_0) = 0 \ \textup{and} \ \beta_k(a) <0 \}$, then 
\begin{equation}\label{eq:kernel} \dim\ker D_{u_0} = \dim \ker D_u + Q^+ + Q^- \quad \textup{for} \ u \in (a, u_0).
\end{equation}
Since \[ \lim_{\lambda \to 0^\pm} \frac{2}{\sqrt \pi} \int_\varepsilon^\infty \lambda e^{-t^2 \lambda^2} dt = \pm 1,  \]
it follows that 
\begin{equation}\label{eq:limit}
 \lim_{u\to u_0^-} \eta_\varepsilon (D_u) = \eta_\varepsilon(D_{u_0}) + Q^+ - Q^-. 
\end{equation}
Recall that by definition $\sflow (D_a, D_{u_0}) = Q^-$ and 
\[ \reta_\varepsilon(D_{u_0}) = \frac{\eta_\varepsilon (D_{u_0}) + \dim\ker (D_{u_0})}{2}. \]
Therefore, the difference of equation $\eqref{eq:kernel}$ and equation $\eqref{eq:limit}$ gives 
\[  \sflow (D_a, D_{u_0}) =  \reta_\varepsilon (D_{u_0})- \lim_{u\to u_0^-} \reta_\varepsilon (D_u). \]
Similarly, $\sflow (D_{u_0}, D_b) = \lim_{u\to u_0^+} \reta_\varepsilon (D_u) - \reta_\varepsilon (D_{u_0}).$
Thus we have
\[ \sflow (D_a, D_b) = \lim_{u\to u_0^+} \reta_\varepsilon (D_u)  - \lim_{u\to u_0^-} \reta_\varepsilon (D_u). \]

With the above results combined, we have the following proposition.
\begin{proposition}\label{prop:sflow}
\[ \sflow(D, g^{-1} D g) = \lim_{\varepsilon \to \infty}\frac{\varepsilon}{\sqrt \pi}\int_0^1 \btr( \dot{D}_u  e^{-\varepsilon^2 D_u^2}) du \]
\end{proposition}
\begin{proof}
Let us cover $[0, 1]$ by finitely many pointed gap intervals $[a_i, b_i]$, $0\leq i \leq n$, with $u_i \in [a_i, b_i]$ such that $b_i = a_{i+1}$ with $u_0 = a_0 = 0$, $u_n = b_n=1$ and $u_j \in (a_j, b_j)$ for $1 \leq j \leq n-1$. Then
\begin{align*}
	\sflow(D_0, D_1) &=  \reta_\varepsilon (D_1) - \reta_\varepsilon (D_0 ) + \sum_{i} \lim_{u\to u_i^+} \reta_\varepsilon (D_u) - \lim_{u\to u_i^-} \reta_\varepsilon (D_u)\\
& = \reta_\varepsilon(D_1) - \reta_\varepsilon(D_0) - \frac{1}{2}\int_0^1 \frac{d}{du}\eta_\varepsilon(D_u) du \\
 &= \reta_\varepsilon(D_1) - \reta_\varepsilon(D_0) +\frac{\varepsilon}{\sqrt \pi} \int_0^1 \btr( \dot{D}_u  e^{-\varepsilon^2 D_u^2}) du - \frac{1}{2}\int_0^1 E_\varepsilon(u) du.
\end{align*}
Notice that $\reta_\varepsilon (g^{-1} Dg) = \reta_\varepsilon(D)$ and  $\int_0^1 E_\varepsilon(u) du$ vanishes when $\varepsilon \to \infty$, hence follows the proposition.
\end{proof}

\section{Large Time Limit}\label{sec:large}
In this section, we prove the equality
\[2\, \sflow(D,  g^{-1} D g ) = \lim_{t\to \infty} \langle \bch^\bullet(t\twD), \Ch_\bullet(g) \rangle. \]
This is the last step remaining to prove Theorem $\ref{thm:main}$. We follow rather closely Getzler's proof for closed manifolds \cite{EG93}.

Recall that we have
\[ \twD = i \begin{pmatrix} 0 & D \\ D & 0 \end{pmatrix} \in \bpo^1(\widehat M; \mathcal S_1)\]
and 
$ p =  \begin{pmatrix} 0 & -g^{-1} \\ g & 0 \end{pmatrix} \in C^\infty_{\textup{exp}}(\widehat M)\otimes \End(\mathbb C^{r|r})$ with $\mathbb C^{r|r} = (\mathbb C^r)^+\oplus (\mathbb C^r)^-$ being $\mathbb Z_2$-graded. 
Let us put 
  \[ \twD_u = (1-u )  \twD + u p \twD p \in \bpo^1(\widehat M; \mathcal S_1\otimes_s \mathbb C^{r|r})\] 
  for $u \in [0, 1]$, where $\mathcal S_1\otimes_s \mathbb C^{r|r}$ is the super-tensor product of $\mathcal S_1$ and $\mathbb C^{r|r}$.
  
We denote $\twD_{u,s} =  \twD_u + sp $ (resp. $\ptwD_{u,s} =  \ptwD_u + sp $), where $(u, s) \in [0, 1]\times (-\infty, 0]$. Consider  the superconnections $ \bcon = d + \twD_{u, s} $ and $\pcon = d + \ptwD_{u, s},$ where $d$ is the standard de Rham differential on the parameter space $[0, 1]\times (-\infty, 0]$.
We have 
\[  \bcon^2 = \twD_{u}^2 + s[\twD_{u}, p] - s^2 +du \dot{\twD}_u + ds p ,\]
\[ \pcon^2 = \ptwD_{u}^2 + s[\ptwD_{u},\bp p] - s^2 +du \dot{\ptwD_u} + ds \bp p .\]
Recall that 
\[ \btr [ Q, K] = \frac{-1}{2\pi i} \int_{-\infty }^{\infty} \ptr\left(\frac{d \indo(Q, \lambda)}{d\lambda} \indo(K, \lambda) \right) d\lambda \]
if either $Q$ or $K$ is in $\Psi^{-\infty}_b (\widehat M, \mathcal V)$, where $\indo(Q, \lambda)$ (resp. $\indo(K,  \lambda)$) is the indicial family of $Q$ (resp. $K$), cf. \cite[Theorem 2.5]{PL05}. A straightforward calculation shows that  
\[ I(D, \lambda) =  \bp D + i \lambda c(\nu) \]
where $\nu = dx$ is the normal cotangent vector  \cite[Proposition 5.4]{EG93}. Therefore, we have
\begin{align*}
\indo(\twD_{u, s }, \lambda) & = \ptwD_u - \lambda c(\nu) + s \bp p,\\
\indo(d \twD_{u, s}, \lambda) & = du \dot{\ptwD_u} + ds \bp p, \\
\indo(\twD_{u, s}^2, \lambda) & = \ptwD_u^2 -\lambda^2 + s[\ptwD_u, \bp p] - s^2.
\end{align*}
Consider the Chern character of $\bcon$,  defined by 
\[ \Ch(\bcon) := \bstr_{(1)}(e^{\bcon^2}). \]
Denote $\Gamma_u$  the contour $\{ u \}\times [0, \infty) $ and $\gamma_s$ the contour $[0, 1]\times \{ s \}$. By Stoke's theorem, we have
\begin{equation}\label{eq:stoke}
	  \int_{\Gamma_1} \Ch(\bcon) - \int_{\Gamma_0} \Ch(\bcon) + \int_{\gamma_0} \Ch(\bcon) - \lim_{s\to \infty}\int_{\gamma_s} 
\Ch(\bcon) = \int_{[0, 1]\times [0, \infty)} d\Ch(\bcon). 
\end{equation}

\subsection{Technical Lemmas}
In this section, let us prove several technical lemmas. Notice that by definition, we have 
\begin{align*}
 \Ch(\bcon) &= du\int_0^1\bstr_{(1)} \left( e^{\sigma(\twD_u^2 + s[\twD_u, p]-s^2)}\dot  \twD_u e^{(1-\sigma)(\twD_u^2 + s[\twD_u, p]-s^2)} \right) d\sigma \\
 &\quad + ds\int_0^1\bstr_{(1)} \left(e^{\sigma(\twD_u^2 + s[\twD_u, p]-s^2)}p\ e^{(1-\sigma)(\twD_u^2 + s[\twD_u, p]-s^2)}\right)d\sigma 
\end{align*}

\begin{lemma}\label{lemma:sflow}
	\begin{align*}
		 \int_{\gamma_0} \Ch(\bcon) &= \int_0^1 \bstr_{(1)}(\dot\twD_u e^{\twD^2_u}) du	
	\end{align*}
\end{lemma}
\begin{proof}
Since
\begin{align*}
&\quad\bstr_{(1)}[e^{\sigma \twD^2_u}, \dot \twD_u e^{(1-\sigma)\twD_u^2}] \\
&= \int_{-\infty}^\infty\pstr_{(2)} \left(\frac{d(e^{\sigma(\ptwD^2 -\lambda^2)})}{d\lambda}  \dot \ptwD_u e^{(1-\sigma)(\ptwD^2-\lambda^2)}\right) d\lambda = 0,
\end{align*}
it follows that 
\begin{align*}
	\int_{\Gamma_0} \Ch(\bcon) &= \int_0^1 du\left(\int_0^1  \bstr_{(1)}(e^{\sigma \twD^2_u}\dot\twD_u e^{(1-\sigma)\twD^2_u})d\sigma\right)  \\
	& = \int_0^1 \bstr_{(1)}(\dot\twD_u e^{\twD^2_u}) du
\end{align*}

\end{proof}

\begin{lemma}\label{lemma:tech1}
	\[  \lim_{s\to \infty}\int_{\gamma_s} \Ch(\bcon) = 0 \]
\end{lemma}
\begin{proof} 
First notice that a similar argument as that in Lemma $\ref{lemma:sflow}$ shows that 
 \begin{align*} 
 &\int_0^1\bstr_{(1)} \left(e^{\sigma(\twD_u^2 + s[\twD_u, p]-s^2)}p\ e^{(1-\sigma)(\twD_u^2 + s[\twD_u, p]-s^2)}\right)d\sigma\\
 & = \bstr_{(1)} (p\ e^{\twD_u^2 + s[\twD_u, p]-s^2}).
 \end{align*}
 Now we also have 
\begin{align*}
	&\bstr_{(1)} (p\ e^{\twD_u^2 + s[\twD_u, p]-s^2} )\\
	& = \sum_{n=0}^{\infty} e^{-s^2}  s^n\int_{\Delta^n} \bstr_{(1)} \left(p \ e^{\sigma_0 \twD_u^2} [\twD_u, p]e^{\sigma_1 \twD_u^2} \cdots [\twD_u, p] e^{\sigma_n \twD_u^2} \right) d\sigma.
\end{align*}
The estimates in Section $\ref{subsec:bound}$ show that  
\begin{align*}
 &\left| \int_{\Delta^n} \bstr_{(1)} \left(\dot \twD_u e^{\sigma_0 \twD_u^2} [\twD_u, p]e^{\sigma_1 \twD_u^2} \cdots [\twD_u, p] e^{\sigma_n \twD_u^2} \right) d\sigma\right|\\
 & \leq 2^n (n+1)\frac{\mathcal K_1^n + 2\mathcal K_0}{n!} \ \bnorm p\|^{n+2}.
\end{align*}
for some constants $\mathcal K_0$ and $\mathcal K_1$. In fact  $\mathcal K_0$ and $\mathcal K_1$ can be chosen independent of $u$, since there is constant $\mathcal C$ such that 
\[\left|\Tr(e^{\sigma\twD_{\mathbb R}^2} - e^{\sigma\twD_u^2})|_{(-\infty, 0]\times \partial M}\right| \leq \mathcal C\] for all $\sigma, u\in [0, 1]$ (cf.  \cite[Proposition 3.1]{L-M-P09}). Hence 
\[
\left|\bstr_{(1)} (p\ e^{\twD_u^2 + s[\twD_u, p]-s^2} )\right| \leq \mathcal K'  e^{-s^2 + 2\mathcal K\, \bnorm p \| s } 
\]
for some constants $\mathcal K$ and $\mathcal K'$. Therefore $\int_{\gamma_s} \Ch(\bcon)= O(e^{-s^2/2}) $ as $s\to \infty$, hence follows the lemma.

\end{proof}

\begin{lemma}\label{lemma:tech2}
	\[  \int_{\Gamma_0} \Ch(\bcon) =  -\int_{\Gamma_1} \Ch(\bcon) = \frac{1}{2} \langle \bch^\bullet(\twD), \sum_{k=0}^{\infty} k! \Str(p, \cdots, p)_{2k+1} \rangle\]
\end{lemma}
\begin{proof}
When $u = 0 $, we have $ \bcon^2 = \twD^2 + s[\twD, p] - s^2 + ds p .$ Using Duhamel's principle, we see that 
\begin{align*}
	\Ch(\bcon) & = \sum_{k=0}^{\infty} s^{2k+1}e^{-s^2} ds \sum_{i=0}^{2k+1} \langle 1, [\twD, p], \cdots, [\twD, p], p, [\twD, p], \cdots, [\twD, p]\rangle  \\
	 & = \sum_{k=0}^{\infty} s^{2k+1}e^{-s^2} ds \sum_{i=0}^{2k+1} \langle p, [\twD, p], \cdots, [\twD, p]\rangle.
\end{align*}
Hence 
\[ \int_{\Gamma_0} \Ch(\bcon) =  \frac{1}{2} \langle\bch^\bullet (\twD) , \sum_{k=0}^{\infty} k! \Str(p, \cdots, p)_{2k+1} \rangle.\]
When $u= 1$, we have $ \bcon^2 = \twD_{1}^2 + s[\twD_{1}, p] - s^2 + ds p $. Notice that $[\twD_1, p] = - [\twD, p]$ and 
\[ \twD_1 = \begin{pmatrix} g^{-1}\twD g  & 0  \\ 0 & g\twD g^{-1} \end{pmatrix}.
 \]
It is straightforward to verify that 
\begin{align*} 
	&[\twD_1, p]e^{\sigma \twD_1^2} [\twD_1, p] e^{\tau \twD_1^2} \\
	&=\begin{pmatrix} g^{-1}[\twD, g] e^{\sigma \twD^2}[\twD, g^{-1}] e^{\tau \twD^2} g & 0 \\ 0 & g [\twD, g^{-1}] e^{\sigma \twD^2}[\twD, g] e^{\tau \twD^2} g^{-1}\end{pmatrix}
\end{align*}		
and 
\begin{align*} 
	p e^{\sigma \twD_1^2} [\twD_1, p] e^{\tau \twD_1^2} &=\begin{pmatrix}  e^{\sigma \twD^2}[\twD, g^{-1}] e^{\tau \twD^2} g & 0 \\ 0 & e^{\sigma \twD^2}[\twD, g] e^{\tau \twD^2} g^{-1}\end{pmatrix}.
\end{align*}
It follows that 
\begin{align*} 
	\int_{\Gamma_1} \Ch(\bcon) & = \frac{1}{2}\langle \bch^\bullet(\twD), \Ch_\bullet(g^{-1}) - \Ch_\bullet(g) \rangle. \\
	& = -\frac{1}{2} \langle\bch^\bullet (\twD_1) , \sum_{k=0}^{\infty} k! \Str(p, \cdots, p)_{2k+1} \rangle  
\end{align*}
\end{proof}

\begin{lemma}
 \[ d\Ch(\bcon) = -\bstr[\twD_{u, s} , e^{\bcon^2} ] =  -\pstr_{(2)} ( e^{\pcon^2})  \]
\end{lemma}
\begin{proof}
Since $[\bcon, e^{\bcon^2}] = 0$, 
\[ d\, \bstr_{(1)}(e^{\bcon^2}) = \bstr_{(1)}[d , e^{\bcon^2}] = - \bstr_{(1)}[ \twD_{u, s}, e^{\bcon^2}]. \]
Therefore
\begin{align*}
d\Ch(\bcon) &= \frac{1}{2\pi i} \int_{-\infty }^{\infty} \pstr_{(1)}\left(\frac{d \indo(\twD_{u, s}, \lambda)}{d\lambda} \indo(e^{\bcon^2}, \lambda) \right) d\lambda \\
& = \frac{1}{2\pi i } \int_{-\infty}^{\infty} \pstr_{(1)}\left( -c(\nu) \indo(e^{\bcon^2}, \lambda)\right) d\lambda\\
& = \frac{-1}{\sqrt\pi }\int_{-\infty}^{\infty} \pstr_{(2)}\left( \indo(e^{\bcon^2}, \lambda)\right)d\lambda\\
& =  \frac{-1}{\sqrt\pi }\int_{-\infty}^{\infty} e^{-\lambda^2 } d\lambda \  \pstr_{(2)} ( e^{\pcon^2} )\\
& =  -\pstr_{(2)}( e^{\pcon^2})
\end{align*}

\end{proof}


By Duhamel's principle, the $2$-form components in $\pstr_{(2)}( e^{\pcon^2})$ can be expanded as
\begin{align}
	&\sum_{k=2}^\infty \sum_{ 1 \leq i< j \leq k } s^{k-2} e^{-s^2} \langle 1,[\ptwD_u , p], \cdots, [\ptwD_u, p], \underbrace{p}_{i-\textup{th}} , [\ptwD_u, p], \cdots, \notag\\
	& \hspace{5cm} [\ptwD_u, p], \underbrace{\dot{\ptwD_u}}_{j-\textup{th}},  [\ptwD_u, p], \cdots, [\ptwD_u, p] \rangle \diff u \diff s \notag \\ 
	 & + \sum_{k=2}^\infty \sum_{ 1 \leq i< j \leq k } -  s^{k-2} e^{-s^2}  \langle 1,[\ptwD_u , p], \cdots, [\ptwD_u, p], \underbrace{\dot{\ptwD_u}}_{i-\textup{th}} , [\ptwD_u, p], \cdots, \notag\\
	& \hspace{5cm} [\ptwD_u, p], \underbrace{p}_{j-\textup{th}},  [\ptwD_u, p], \cdots, [\ptwD_u, p] \rangle \diff u \diff s \label{Eq:expansion}
\end{align}

Recall that (cf.\cite[Lemma 2.2]{EG-AS89})
\[  \langle A_0, \cdots, A_n \rangle = \sum_{i=0}^{n}(-1)^{(|A_0| + \cdots + |A_{i-1}|)(|A_i|+\cdots+|A_n|)}\langle 1,  A_i, \cdots, A_n, A_0, \cdots, A_{i-1} \rangle \]
Since $ \ptwD_u$, $\dot{\ptwD_u}$ and $p$ are of odd degree and $ [\ptwD_u, p]$ is of even degree, one has
\[ \eqref{Eq:expansion}  =  \sum_{k=2}^\infty \sum_{i=1}^{k-1}  s^{k-2} e^{-s^2}\langle p, [\ptwD_u, p], \cdots, [\ptwD_u, p], \underbrace{\dot{\ptwD_u}}_{i-\textup{th}}, [\ptwD_u, p], \cdots, [\ptwD_u, p] \rangle\diff u \diff s . \] 
Let us define  
\begin{equation}
\sch^n(\ptwD_u, V) (a_0, \cdots, a_n) = \iota(V) \langle a_0, [\ptwD_u, a_1], \cdots, [\ptwD_u, a_n] \rangle
\end{equation}
where 
\[	\iota(V)\langle A_0, \cdots, A_n \rangle := \sum_{0\leq i \leq n} (-1)^{|V| (|A_0|+\cdots+|A_i|)} \langle A_0, \cdots, A_i, V, A_{i+1}, \cdots, A_n \rangle\]
Then the calculation above shows that 
\begin{align*}
&\pstr_{(2)}( e^{\pcon^2}) = -\sum_{k=0}^\infty s^{k} e^{-s^2}\langle\sch^{k}(\ptwD_u, \dot{\ptwD_u}), \ \Str(p, \cdots, p)_{k}\rangle \diff u \diff s.
\end{align*}
We summarize this in the following lemma.
\begin{lemma}
\[ \int_{[0,1]\times [0, \infty)} \pstr_{(2)}( e^{\pcon^2}) = -\frac{1}{2}  \langle \int_{0}^1\sch^\bullet(\ptwD_u, \dot{\ptwD_u})\diff u, \ \sum_{k=0}^{\infty} k! \ \Str(p, \cdots, p)_{2k+1}\rangle \]
\end{lemma}
\begin{proof}
Notice that $\Str(p, \cdots, p)_{k} = 0 $ for $k$ even, and 
\[\int_{0}^{\infty} s^{2k+1}e^{-s^2}\diff s = \frac{k!}{2}.\]

\end{proof}

\subsection{Large Time Limit}
Recall that $\bp D$ is invertible and $g\in U_k(C^\infty(N))$ is a unitary such that $\|[\bp D, g]\| < \lambda$ with $\lambda$ the lowest nonzero eigenvalue of $|\bp D|$. Notice that 
\[ \|[\bp D, g^{-1}]\| = \|-g^{-1}[\bp D, g]g^{-1}\| \leq   \|[\bp D, g]\|,\]
and similarly $\|[\bp D, g]\| \leq \|[\bp D, g^{-1}]\|$. Hence $\|[\bp D, g^{-1}]\| = \|[\bp D, g]\| $.  

\begin{lemma}\label{lemma:tech3}
Let $\bcon_t = d+ t\twD_{u, s}$.  Then
	\[  \lim_{t\to \infty}\int_{[0, 1]\times [0, \infty)} d\Ch(\bcon_t) = 0. \]
\end{lemma}
\begin{proof} 
Notice that
\begin{align} 
|\bp D + ug^{-1}[\bp D, g]| &\geq |\bp D|-u\|g^{-1}[\bp D,g]\| \geq \lambda - u \|[\bp D,g]\| \\
|\bp D + ug[\bp D, g^{-1}]| &\geq \lambda - u \|[\bp D,g]\| 
\end{align}
When $u=1$, $\bp D + u g^{-1}[\bp D, g]= g^{-1}\, \bp D g $. The lowest eigenvalue of $|g^{-1}\, \bp D g|$ is also $\lambda$, since $g$ is a unitary. Therefore, a similar argument as above shows	 that  
\begin{align}
	 |\bp D + ug^{-1}[\bp D, g]| &\geq \lambda - (1-u)\|[\bp D, g]\|,\\
 |\bp D + ug[\bp D, g^{-1}]| &\geq \lambda - (1-u)\|[\bp D, g]\|.
\end{align}
Thus $|\ptwD_u|$ is bounded below by 
\[\lambda_u := \max\{\lambda- u\|[\bp D, g]\|, \lambda- (1-u)\|[\bp D, g]\|\}.\]
Then there exists a constant $C$ such that
  \[\Tr (e^{t^2\ptwD_u^2}) \leq C e^{-t^2  \lambda_u^2} \] for all $t\geq 1,$
where we may take $C = \sup_{u\in [0,1]} e^{ \lambda_u^2}\Tr(e^{\ptwD_u^2}),$ 
cf. \cite[Theorem C]{EG-AS89}. One also notices that $\|[\ptwD_u, p]\|= (1-2u)\|[\ptwD, p]\|  < \lambda_u$. Therefore we have
\begin{align*}
	&\left|\langle p, [\ptwD_u, p], \cdots, [\ptwD_u, p], \dot{\ptwD_u}, [\ptwD_u, p], \cdots, [\ptwD_u, p] \rangle_{2k+1}\right| \\
	&\leq   \frac{1}{(2k)!}\Tr (e^{t^2\ptwD_u^2}) \|p\|\cdot \|[\ptwD_u, p]\|^{2k-1} \|\dot{\ptwD_u}\|\\
    & \leq \frac{C}{ (2k)! } e^{-t^2\lambda_u^2} \|[\ptwD_u, p]\|^{2k} \|p[\ptwD, p]\|.
\end{align*}
Hence
\begin{align*}
&\left|\int_{[0, 1]\times [0, \infty)} d\Ch(\bcon_t)\right| = \left| \int_{[0,1]\times [0, \infty)} \pstr_{(2)}( e^{\pcon^2_t})\right|\\
& = \left|\langle \int_{0}^1\sch^\bullet(t\, \ptwD_u, t\, \dot{\ptwD_u})\diff u, \ \sum_{k=0}^{\infty} k! \ \Str(p, \cdots, p)_{2k+1}\rangle\right|\\
&\leq \int_0^1 \sum_{k=1}^\infty	 \frac{C}{ k! } e^{-t^2\lambda_u^2} ( \|[\ptwD_u, p]\| \cdot t)^{2k} \|p[\ptwD, p]\| du\\
& \leq C \|p[\ptwD, p]\| \int_0^1 e^{( \|[\ptwD_u, p]\|^2 -\lambda^2_u) t^2} du \\
\end{align*}
where the last term goes to zero when $t\to \infty$. This finishes the proof.

\end{proof}

\begin{proposition}\label{prop:large}
\[ 2 \,\sflow(D,  g D g^{-1} ) = \lim_{t\to \infty} \langle \bch^\bullet(t\twD), \Ch_\bullet(g) - \Ch_\bullet(g^{-1})\rangle. \]
\end{proposition}
\begin{proof}
Notice that
	\begin{align*}
		\twD_u &= \begin{pmatrix} \twD + u g^{-1}[\twD, g]  & 0  \\ 0 & \twD + u g[\twD, g^{-1}] \end{pmatrix} \\
		& = i\begin{pmatrix} 0 & D + ug^{-1}[D, g] &  &  \\ D + ug^{-1}[D, g] & 0 &  &  \\
		                    &  & 0 & D + ug[D, g^{-1}] \\  &  & D + ug[D, g^{-1}] & 0 \end{pmatrix}.
	\end{align*}
Hence 
\begin{align*}
\int_0^1 \bstr_{(1)}(\dot\twD_u e^{\twD^2_u}) du &= \frac{1}{\sqrt \pi}\int_{0}^1 \btr (g^{-1}[D, g] e^{- (D+ug^{-1}[D, g])^2 }) du \\
		& \quad \quad -\frac{1}{\sqrt \pi}  \int_{0}^1 \btr (g[D, g^{-1}] e^{- (D + ug[D, g^{-1}])^2 }) du.
\end{align*}
It follows from Theorem $\ref{prop:sflow}$ that 
	\[ \lim_{t\to \infty}\int_0^1 \bstr_{(1)}(t \dot\twD_u e^{t^2\twD^2_u}) du  = \sflow(D, g^{-1} D g) - \sflow(D, g D g^{-1}). \]
Now by applying Lemma $\ref{lemma:tech1} $, $\ref{lemma:tech2}$ and $\ref{lemma:sflow}$ to equation $\eqref{eq:stoke}$, we have  	
	\begin{align*} \int_0^1 \bstr_{(1)}(t \dot\twD_u e^{t^2\twD^2_u}) du & =  \langle \bch^\bullet(t \twD), \Ch_\bullet(g) - \Ch_\bullet(g^{-1})\rangle \\
		 & -\frac{1} {2}  \langle \int_{0}^1\sch^\bullet(t \ptwD_u, t \dot{\ptwD_u})\diff u, \ \sum_{k=0}^{\infty} k! \ \Str(p, \cdots, p)_{2k+1}\rangle.
	\end{align*}
Since $\sflow(D, g D g^{-1}) = -\sflow(D, g^{-1} D g)$, it follows from  Lemma $\ref{lemma:tech3}$ that 
 \[ 2\, \sflow(D, g^{-1}Dg)  = \lim_{t\to \infty} \langle \bch^\bullet(t\twD), \Ch_\bullet(g) - \Ch_\bullet(g^{-1}) \rangle \]

\end{proof}


\end{document}